\RequirePackage{fix-cm}
\documentclass{svjour3}                     
\smartqed  
\usepackage{graphicx}
\usepackage{mathptmx}      
\usepackage{latexsym}
\usepackage{amsfonts}
\usepackage{amsmath}
\usepackage{epstopdf}
\usepackage{algorithmic}
\usepackage[bookmarks, a4paper, colorlinks=true]{hyperref}
\newtheorem{thm}{Theorem}

\newtheorem{prop}{Proposition}
\newtheorem{coro}{Corollary}
\newtheorem{rem}{Remark}
\journalname{}
\begin{document}

\title{A stochastic Gronwall's inequality in random time horizon and its application to BSDE
}


\author{Hun O \and
        Mun-Chol Kim   \and
        Chol-Kyu Pak   
}


\institute{Chol-Kyu Pak (corresponding author), \at
              Faculty of Mathematics, Kim Il Sung University, Pyongyang, DPR Korea.\\
              \email{pck2016217@gmail.com}           
           \and
}

\date{Received: date / Accepted: date}

\maketitle

\begin{abstract}
In this paper, we introduce and prove a stochastic Gronwall's inequality in (unbounded) random time horizon. As an application, we prove a comparison theorem for backward stochastic differential equation (BSDE for short) with random terminal time under stochastic monotonicity condition.
\keywords{Gronwall's inequality, stochastic, random time horizon, backward stochastic differential equation, comparison}
\subclass{MSC 60E15 \and MSC 60H20 }
\end{abstract}

\section{Introduction}
Gronwall's inequality is a handy tool to derive many useful results such as uniqueness, comparison, boundness, continuous dependence and stability in the theory of differential and integral equations.  It was first introduced by Gronwall \cite{Gronwall} as a differential form and the integral inequality was proposed by Bellman \cite{bellman}. Since then, many researchers studied the various types of generalizations of this inequality motivated by the development of the differential and integral equations (\cite{Bihari}, \cite{Chen}, \cite{Pachpatte}).
Among such generalizations, we are concerned with the stochastic version of Gronwall's inequality. 
\par Let $(\Omega,\mathcal {F},\mathbb {P})$ be a complete probability space on which a $d$-dimensional Brownian motion $B=(B_t)_{t\ge 0}$ is defined. Let $(\mathcal F_t)_{t\ge 0}$ be the right-continuous completion of the natural filtration generated by $B$, that is,
$\mathcal F_t:=\sigma\{B_s,s\le t\}$ and argument them by $\mathbb P$-null sets. 
\par In Wang and Fan \cite{Wang}, they first proved the following stochastic Gronwall's inequality.  

\begin{prop}\label{prop:1}
Let $c>0$, $T>0$ and an $(\mathcal F_t )_{t\ge 0}$-progressive measurable process $a:\Omega\times [0,T]\rightarrow\mathbb R^+$ satisfy $\int_0^T\!a(t)\,dt\le M,\ \mathbb P-a.s.$ for some constant $M>0$.
If an $(\mathcal F_t)_{t\ge 0}$-progressive measurable process $x:\Omega\times [0,T]\rightarrow\mathbb R^+$ satisfies 
\[
\mathbb E \bigg[\sup_{t\in [0,T]}⁡x(t)\bigg]<+\infty,\  x(t)\le c+\mathbb E\bigg[\int_t^T\!a(s)x(s)\,ds\bigg |\mathcal F_t \bigg],\ t\in [0,T].
\]
then, for each $t\in [0,T]$,
\[
x(t)\le c\cdot \mathbb E[e^{\int_t^T\!a(s)\,ds}|\mathcal F_t ],\ \mathbb P-a.s.
\]
\end{prop}
If the random processes $a(t)$ and $x(t)$ are deterministic functions in above proposition, then we reach the well-known Gronwall's inequality as follows.
\begin{coro}
If $a(t)$ and $x(t)$ are two non-negative (deterministic) functions defined on $[0,T]$ which  satisfy
\[
x(t)\le c+\int_t^T\!a(s)x(s)\,ds,\ t\in [0,T].
\]
then, for each $t\in [0,T]$,
\[
x(t)\le c\cdot e^{\int_t^T\!a(s)\,ds}.
\]
\end{coro}
In this paper, we study the complete version of Gronwall's inequality in stochastic sense.  
More precisely, in above Proposition \ref{prop:1}, the constants $c$ and $T$ are replaced by a random variable and (unbounded) stopping time, respectively and the integral of $a(t)$ is not assumed to be essentially bounded.
~\\ We use methods by the martingale representation and random time change to prove the main inequality. 
Due to the type of the proposed inequality, the application to the stochastic differential (or integral) equation with stochastic coefficients defined up to random time (more precisely, stopping time) is naturally considered. We give the proof of a comparison theorem of $L^p$-solutions to backward stochastic differential equation (BSDE for short) with random terminal time and stochastic coefficients by using the stochastic Gronwall's inequality in random time horizon, effectively.

\section{Notations}\label{sec:2}
Let $p>1$ and $\tau$ be an $(\mathcal F_t)_{t\ge 0}$-stopping time. That is,  $\forall t\ge 0, \{\omega|\tau(\omega)\le t\}\in\mathcal F_t$.
Throughout all the paper, $|\cdot|$ means the standard uclidean norm. 
We put $A(t):=\int_0^t\!a(s)\,ds$, where $a(s)$ is a non-negative progressive-measurable process. The symbols $\mathbb E[\cdot]$ and $\mathbb E[\cdot|\mathcal F_t]$ denote the expectation and conditional expectation (with respect to $\mathcal F_t$), respectively.  
\begin{itemize}	
    \item $L_{\theta}^p (\mathcal F_{\tau})$ is the set of real-valued $\mathcal F_{\tau}$-measurable random variables $\xi$ such that 
                \[
                     \mathbb E[e^{p/2 \theta A(\tau)} |\xi|^p]<+\infty. 
                \]
    \item $H_{\tau,\theta}^p (\mathbb R)$ is the set of real-valued c\`adl\`ag, adapted processes $Y$ such that 
               \[
                    \mathbb E\bigg[\sup_{0\le t\le\tau}⁡e^{p/2 \theta A(t)} |Y_t|^p\bigg]<\infty.
               \]
    \item $H_{\tau,\theta}^{p,a}(\mathbb R)$ is the set of real-valued c\`adl\`ag, adapted processes $Y$ such that
               \[
                   \mathbb E\bigg[\bigg(\int_0^\tau\! a(t) e^{\frac{\theta}{2} A(t)} |Y_t|\,dt\bigg)^p\bigg]<\infty.
               \]
    \item $M_{\tau,\theta}^p (\mathbb R^{1\times d})$ is the set of predictable processes $Z$ with values in $\mathbb R^{1\times d}$ such that 
	    \[
	        \mathbb E\bigg[(\int_0^\tau\!e^{\theta A(t)} |Z_t|^2\,ds)^{p/2}\bigg]<\infty.
              \]
    \item For $m,n\in\mathbb R$, $m\wedge n:=\min\{m,n\}$ and $m^+:=\max\{m,0\}$.
\end{itemize}
\section{Main inequality}\label{sec:3}

\begin{thm}\label{thm:1}
Let $p>1,l\ge 0$ be constants and $q$ be a constant such that $1/p+1/q=1$. Let $\tau$ be an $(\mathcal F_t)_{t\ge 0}$-stopping time and $\xi$ be a non-negative random variable. Let $a(t)$ and $x(t)$ be non-negative progressive measurable processes. Assume that $a(t)>\epsilon$ for some constant $\epsilon>0$ and $x(t)$ belongs to $H_{\tau,\theta}^{p,a} (\mathbb R)$. We further assume that $\xi$ belongs to  $L_{\theta}^p (\mathcal F_{\tau})$ for some constant $\theta \ge 0$ satisfying $\mathbb E[e^{\frac{q}{2} (2l-\theta)^+A(\tau))}]<\infty$, where $A(t):=\int_0^t\!a(s)\,ds$.
~\\ If $x(t)\le \mathbb E[\xi+l\int_{t\wedge\tau}^{\tau}\!a(s)x(s)\,ds|\mathcal F_t],\ \mathbb P-a.s.$, then, we have $\mathbb P-a.s.$,
	\[
	x(t)\le\mathbb E[\xi\cdot e^{l\int_{t\wedge\tau}^{\tau}\!a(s)\,ds}|\mathcal F_t].
	\]
\end{thm}

\begin{proof}
Define the process $X(t):=\mathbb E[\xi+l\int_{t\wedge\tau}^{\tau}\!a(s)x(s)\,ds|\mathcal F_t]$, then it follows from the assumptions that $x(t)\le X(t)$.
Let us put $\eta:=\xi+l\int_0^\tau\!a(s)x(s)\,ds$. Then,
\[
\mathbb E\bigg[\bigg(\int_0^\tau\!a(s)x(s)\,ds\bigg )^p\bigg]\le\mathbb E\bigg [\bigg(\int_0^\tau\!a(s) x(s) e^{\theta A(s)/2}\,ds\bigg)^p\bigg]<\infty.
\]
Therefore, 
\[
\mathbb E [|\eta|^p]\le 2^{p-1} \bigg (\mathbb E[|\xi|^p]+l^p \mathbb E\bigg [\bigg (\int_0^\tau\!a(s)x(s)\,ds\bigg )^p\bigg ]\bigg)<\infty.
\]
That is, $\eta \in L_0^p (\mathcal F_\tau)$.
By the martingale representation theorem (see \cite{Pardoux}, page 116, Theorem 2.42), there exists a process $Z$ satisfying $Z\in M_{T,0}^p (\mathbb R^{1\times d})$  for any $T>0$ such that $\mathbb P-a.s.$
\[
\mathbb E[\eta|\mathcal F_t]=\mathbb E[\eta]+\int_0^{t\wedge\tau}\!Z_s\,dB_s.
\]
So,
\[
X(t)=\mathbb E[\eta|\mathcal F_t]-l\mathbb E\bigg [\int_0^{t\wedge\tau}a(s)x(s)ds|\mathcal F_t\bigg ]=\mathbb E[\eta]+\int_0^{t\wedge\tau}\!Z_s \,dB_s -l\int_0^{t\wedge\tau}\!a(s)x(s)\,ds.
\]

Moreover, we have the backward version: for all $T>0$,
\[
X(t)=X(T\wedge\tau)-\int_{t\wedge\tau}^{T\wedge\tau}\!Z_s \,dB_s +l\int_{t\wedge\tau}^{T\wedge\tau}\!a(s)x(s)\,ds.
\]
with $X(\tau)=\xi$.

Or equivalently,
\begin{equation}\label{eq:1}
X(t)=\xi-\int_{t\wedge\tau}^\tau\!Z_s\,dB_s +l\int_{t\wedge\tau}^\tau\! a(s)x(s)\,ds.
\end{equation}
Now we shall show that $Z\in M_{\tau,\theta}^p (\mathbb R^{1\times d})$.
First, we introduce a certain random time change.
Due to the fact that the process $A_t=\int_0^t\!a(s)\,ds$ is strictly increasing and continuous, we can define its reverse denoted by $C_s:=A^{-1}(s)$. 
Then a family of stopping times, $\{C_s\}, s\ge 0$ is an $(\mathcal F_t)_{t\ge 0}$-random time change (see \cite{Revuz} for systematic study of random time change). Set $\widetilde{\mathcal F}_t:=\mathcal F_{C_t}$, then $(\widetilde{\mathcal F}_t)_{t\ge 0}$ is a new filtration. 
For any adapted process $X$, we assume that $\widetilde X$ means the time-changed process, that is, $\widetilde X=X_{C_t}$. We also define $\widetilde{\tau}:=A(\tau)$ and $W_t:=\int_0^t\!\sqrt{\widetilde a(s)}\,d\widetilde{B}_s $.
Note that $W=(W^1,W^2,...,W^d)$ is a $d$-dimensional $(\widetilde{\mathcal F}_t)_{t\ge 0}$-Brownian motion by Levy's characterization theorem.
In fact, for each $i\in\{1,…,d\}$,
\begin{align}
<W^i,W^i>&=\int_0^t\!\widetilde a(s)\,d<\widetilde B ^i,\widetilde B^i>_s =\int_0^t\!\widetilde a(s)\,d\widetilde {<B^i,B^i>}_s\nonumber\\
&=\int_0^t\!\widetilde a(s)\,dC_s =\int_0^{C_t}\!a(s)\,ds=A(C_t )=t.\nonumber
\end{align}
From the properties of (stochastic) integral with respect to time change,
\[
\int_{t\wedge\tau}^\tau\!a(s)x(s)\,ds=\int_{A(t\wedge\tau)}^{A(\tau)}\!x(C_s)\,ds=\int_{A(t)\wedge\widetilde\tau}^{\tau}\!\widetilde{x}(s)\,ds,
\]
\[
\int_{t\wedge\tau}^{\tau}\!Z_s\,dB_s =\int_{A(t)\wedge\widetilde\tau}^{\widetilde\tau}\!\widetilde Z_s \,d\widetilde B_s =\int_{A(t)\wedge\widetilde\tau}^{\widetilde\tau}\!\widetilde a(s)^{-1/2}\widetilde Z_s\,dW_s.
\]
So, we get the following expression with respect to $(\widetilde{\mathcal F}_t)_{t\ge 0}$.
\[
\widetilde X(t)=\xi-\int_{t\wedge\widetilde\tau}^{\widetilde\tau}\!\widetilde a(s)^{-1/2} \widetilde Z_s\,dW_s +l\int_{t \wedge\widetilde\tau}^{\widetilde\tau}\!\widetilde x(s)\,ds.
\]

According to \cite{Briand}, Proposition 3.2, for all $\widetilde T>0$,

\begin{align}
\mathbb{E}\biggl [\sup_{t\in [0,\widetilde{T}\wedge \widetilde{\tau}]}⁡ e^{\frac{p}{2} \theta t} \widetilde X(t)^p&+\bigg (\int_0^{\widetilde T\wedge \widetilde{\tau}}\!\frac{e^{\theta t}}{\widetilde a(t)}|\widetilde Z_t|^2\,dt\bigg )^{p/2}\biggr]\nonumber\\
&\le c(p)\cdot \mathbb E\biggl [e^{p/2 \theta(\widetilde T\wedge \widetilde{\tau})} \widetilde X(\widetilde T\wedge\widetilde\tau)^p+\bigg (\int_0^{\widetilde T\wedge \widetilde{\tau}}\! l e^{\theta s/2} \widetilde x(s)\,ds\bigg )^p\biggr ]\nonumber
\end{align}
for some constant $c(p)$.
Sending $\widetilde T\rightarrow +\infty$,  Fatou's Lemma ensures that
\[
\mathbb E\bigg [\sup_{t\in [0,\widetilde\tau]}⁡e^{\frac{p}{2} \theta t} \widetilde X(t)^p+\bigg (\int_0^{\widetilde\tau}\!\frac{e^{\theta t}}{\widetilde a(t)} |\widetilde Z_t|^2 \,dt\bigg )^{p/2}\bigg ]\le c(p)\cdot\mathbb E\bigg [e^{p/2 \theta \widetilde\tau} \xi^p+\bigg (l\int_0^{\widetilde\tau}\!e^{\theta s/2} \widetilde x(s)\,ds\bigg )^p\bigg ].
\]

And we see that
\[
\sup_{t\in [0,\widetilde\tau]}⁡e^{p/2 \theta t} \widetilde X(t)^p=\sup_{t\in [0,\tau]}⁡e^{p/2 \theta A(t)} X(t)^p,
\]
\[
\int_0^{\widetilde\tau}\!e^{\theta s/2} \widetilde x(s)\,ds=\int_0^\tau e^{\theta A(s)/2} x(s)\,dA(s)=\int_0^\tau\!a(s)e^{\theta A(s)/2} x(s)\,ds,
\]
\[
\int_0^{\widetilde\tau}\!\frac{e^{\theta t}}{\widetilde a(t)}|\widetilde Z_t|^2\,dt=\int_0^\tau\! \frac{e^{\theta A(t)}} {a(t)} |Z_t|^2\,dA(t)=\int_0^\tau\!e^{\theta A(t)} |Z_t|^2\,dt.
\]
Therefore, we deduce
\begin{align}\label{eq:2}
\mathbb E\bigg [\sup_{t\in [0,\tau]}⁡ e^{p/2 \theta A(t)} &X(t)^p+\bigg (\int_0^\tau\!e^{\theta A(t)} |Z_t |^2\,dt\bigg )^{p/2}\bigg ]\nonumber\\
&\le c(p)\cdot\mathbb E\bigg [e^{p/2 \theta A(\tau)} \xi^p+l^p \bigg (\int_0^\tau\!a(s)e^{\theta A(s)/2}x(s)\,ds\bigg )^p\bigg ]<\infty.
\end{align}
So we have proved that $Z\in M_{\tau,\theta}^p (\mathbb R^{1\times d})$.
Applying Ito's formula to \eqref{eq:1}, we get
\[
 X(t) e^{lA(t\wedge\tau)}=\xi e^{lA(\tau)}+l\int_{t\wedge\tau}^\tau\! a(s) e^{lA(s)} [x(s)-X(s)]\,ds-\int_{t\wedge\tau}^\tau\!e^{lA(s)} Z_s\,dB_s.
\]
From $x(t)\le X(t)$, it follows that
\begin{equation}\label{eq:3}
X(t) e^{lA(t\wedge\tau)} \le\xi e^{lA(\tau)} -\int_{t\wedge\tau}^\tau e^{lA(s)} Z_s\,dB_s.
\end{equation}
From the expression \eqref{eq:2}, Burkholder-Davis-Gundy's inequality (BDG inequality for short) and Young's inequality, we deduce
\begin{align}
\mathbb E\bigg [\sup_{t\ge 0}⁡|\int_0^{t\wedge\tau}\!e^{lA(s)} Z_s\,dB_s |\bigg]&\le c\cdot \mathbb E\bigg [\bigg(\int_0^\tau\!e^{2lA(s)} |Z_s|^2 \,ds\bigg)^{1/2}\bigg]\nonumber\\
&=c\cdot\mathbb E\bigg[\bigg(\int_0^\tau\! e^{\theta A(s)}  |Z_s|^2\cdot e^{(2l-\theta)A(s)}\,ds\bigg)^{1/2}\bigg]\nonumber\\
&\le c\cdot\mathbb E\bigg[e^{(l-\theta/2)A(\tau)\vee 0} \bigg(\int_0^\tau\!|Z_s|^2\cdot e^{\theta A(s)}\,ds\bigg)^{1/2}\bigg]\nonumber\\
&\le \frac{c}{p} \mathbb E\bigg[\bigg(\int_0^\tau\!|Z_s|^2\cdot e^{\theta A(s)}\,ds\bigg)^{p/2}\bigg]+\frac{c}{q}\mathbb E\big[e^{\frac{q}{2} (2l-\theta)^+A(\tau)}\big]<+\infty\nonumber.
\end{align}
Thus, $M(t)=\int_0^{t\wedge\tau}\!e^{lA(s)}z_s\,dB_s$ is the uniformly integrable martingale.
Taking conditional expectations with respect to $\mathcal F_t$ on both sides of \eqref{eq:3}, we get
\[
X(t) e^{lA(t\wedge\tau)}=\mathbb E[X(t) e^{lA(t\wedge\tau)}|\mathcal{F}_t]\le \mathbb E[\xi e^{lA(\tau)}|\mathcal{F}_t].
\]
So, we have $x(t)\le X(t)\le\mathbb E[\xi e^{l\int_{t\wedge\tau}^\tau\! a(s)\,ds)}|\mathcal {F}_t]$ which is the desired result.
\end{proof}
\begin{rem}
If $A(\tau)=\int_0^{\tau}a(s)ds\le M,\ \mathbb P-a.s.$ for some $M\ge 0$, then 
\begin{align}
\mathbb E\bigg[\bigg(\int_0^\tau\! a(t) e^{\frac{\theta}{2} A(t)} x(t)\,dt\bigg)^p\bigg]&\le\mathbb E\bigg[\bigg(A(\tau)\cdot\sup_{0\le t\le\tau}⁡e^{\frac{\theta}{2} A(t)} x(t)\bigg)^p\bigg]\nonumber\\
&\le M^p\cdot\mathbb E\bigg[\sup_{0\le t\le\tau}⁡e^{p/2 \theta A(t)} x(t)^p\bigg]\nonumber.
\end{align}
So, $x(t)\in H_{\tau,\theta}^{p,a} (\mathbb R)$ holds whenever $x(t)\in H_{\tau,\theta}^{p} (\mathbb R)$. Therefore, in Theorem \ref{thm:1}, we can give an alternative assumption such that $x(t)\in H_{\tau,\theta}^{p} (\mathbb R)$, instead of $x(t)\in H_{\tau,\theta}^{p,a} (\mathbb R)$. 
\end{rem}
\section{Application}\label{sec:4}
In this section we show a comparison principle of $L^p$-solutions of BSDEs with random terminal time under stochastic montonicity condition on generator.
Let us consider the following one-dimensional BSDE with random terminal time.
\begin{equation}\label{eq:4}
Y_t=\xi+\int_{t\wedge\tau}^\tau\!f(s,Y_s)\,ds-\int_{t\wedge\tau}^\tau\!Z_s\,dB_s.
\end{equation}
where the terminal time $\tau$ is an $(\mathcal F_t)_{t\ge 0}$-stopping time, the terminal value $\xi$ is an $\mathcal F_{\tau}$-measurable random variable and the generator $f:\Omega\times\mathbb R^+\times\mathbb R\times\mathbb R^{1\times d}\rightarrow\mathbb R$ is $(\mathcal F_t)_{t\ge 0}$-progressive measurable.
The solution of BSDE \eqref{eq:4} is a pair $(Y_t,Z_t )_{t\ge 0}$ of adapted processes such that $Y_t=\xi$ and $Z_t=0$, $\mathbb P-a.s.$ for $t\ge\tau$ and for any $T\ge 0$, 
\[
Y_{t\wedge\tau}=Y_{T\wedge\tau}+\int_{t\wedge\tau}^{T\wedge\tau}\!f(s,Y_s)\,ds-\int_{t∧τ}^{T\wedge\tau}\!Z_s\,dB_s,\ 0\le t\le T.
\]
For the convenience, we characterize the BSDE \eqref{eq:4} by a triple $(\tau,\xi,f)$.
\begin{thm}\label{thm:2}
Let $p>1$ and consider two BSDEs with data $(\tau,\xi^1,f^1)$ and $(\tau,\xi^2,f^2)$. 
Let $(Y^1,Z^1)$ and $(Y^2,Z^2)$ be two solutions of the BSDE \eqref{eq:4} corresponding to $(\tau,\xi^1,f^1)$ and $(\tau,\xi^2,f^2)$, respectively. Suppose that $f$ is stochastic monotone in $y$, that is, there exists a non-negative, progressively process $a(t)$ such that for all $y,y'\in\mathbb R$;
\[
(y-y')(f^1(t,y)-f^1(t,y'))\le a_t (y-y')^2.
\]
We assume that $a_t>\varepsilon$ for some $\varepsilon>0$. Set  $A(t):=\int_0^t\!a_s\,ds$. Let $q$ be a constant such that $1/p+1/q=1$. Suppose that $(Y^i,Z^i), i=1,2$ belongs to $H_{\tau,\theta}^{p,a} (\mathbb R)\times M_{\tau,0}^p (\mathbb R^{1\times d})$ for some $\theta\in\mathbb R$ such that $\mathbb E[e^{\frac{q}{2} (2-\theta)^+A(\tau)}]<+\infty$.
~\\ If $\xi^1\le \xi^2$ and $f^1(t,Y_t^2)\le f^2(t,Y_t^2)$, then $Y_t^1\le Y_t^2$, $\mathbb P-a.s.$
\end{thm}
\begin{proof}
Define $\bar Y:=Y^1-Y^2,\bar Z:=Z^1-Z^2,\bar\xi:=\xi^1-\xi^2$, then
\[
\bar Y_t=\bar\xi+\int_{t\wedge\tau}^\tau\![f^1 (s,Y_s^1)-f^2 (s,Y_s^2)]\,ds-\int_{t\wedge\tau}^\tau\!\bar Z_s\,dB_s.
\]
By the virtue of Ito-Tanaka's formula (see Exercise VI.1.25 in \cite{Revuz} for details), we have
\begin{align}
\bar Y_t^+ &=\bar \xi^+ +\int_{t\wedge\tau}^{\tau}\! \textbf{1}_{\bar Y_s>0} [f^1 (s,Y_s^1)-f^2 (s,Y_s^2)]\,ds-\int_{t\wedge\tau}^{\tau}\! \textbf{1}_{\bar Y_s>0} \bar Z_s\,dB_s-\frac{1}{2}L_t^0\nonumber\\
&\le \bar \xi^+ +\int_{t\wedge\tau}^{\tau}\! \textbf{1}_{\bar Y_s>0} [f^1 (s,Y_s^1)-f^2 (s,Y_s^2)]\,ds-\int_{t\wedge\tau}^{\tau}\! \textbf{1}_{\bar Y_s>0} \bar Z_s\,dB_s,\nonumber
\end{align}
where $L_t^0$ is the local time of semi-martingale $\bar Y_t$, it is increasing and $L_0^0=0$.
~\\
Using this and the assumptions of the theorem, we get
\begin{align}
\bar Y_t^+&\le \bar \xi^+ +\int_{t\wedge\tau}^{\tau}\! \textbf{1}_{\bar Y_s>0} [f^1 (s,Y_s^1)-f^1 (s,Y_s^2)]\,ds-\int_{t\wedge\tau}^{\tau}\! \textbf{1}_{\bar Y_s>0} \bar Z_s\,dB_s\nonumber\\
&=\bar \xi^+ +\int_{t\wedge\tau}^{\tau}\! \textbf{1}_{\bar Y_s>0}\frac{\bar Y_s}{|\bar Y_s|} [f^1 (s,Y_s^1)-f^1 (s,Y_s^2)]\,ds-\int_{t\wedge\tau}^{\tau}\! \textbf{1}_{\bar Y_s>0} \bar Z_s\,dB_s\nonumber\\
&\le \bar \xi^+ +\int_{t\wedge\tau}^{\tau}\! a_s \bar Y_s^+\,ds-\int_{t\wedge\tau}^{\tau}\! \textbf{1}_{\bar Y_s>0} \bar Z_s\,dB_s.\nonumber
\end{align}
Since $\big\{\int_0^{t\wedge\tau}\! \textbf{1}_{\bar Y_s>0} \bar Z_s\,dB_s,\ t\ge 0\big\}$ is a martingale, we obtain
\[
\bar Y_t^+ \le \mathbb E\bigg[\bar\xi^+ +\int_{t\wedge\tau}^\tau\!a_s \bar Y_s^+\,ds\bigg |\mathcal F_t\bigg].
\]
From $\xi^1\le \xi^2$, it follows that $(\bar\xi^+ )^p=0$.
Theorem \ref{thm:1} yields that $\bar \xi^+ =0$. Hence, $Y_t^1\le Y_t^2$.
\end{proof}

\end{document}